\documentclass[10pt]{amsart}
\usepackage{amssymb}
\usepackage{amscd}
\usepackage{graphicx}
\usepackage[all,cmtip]{xy}

\newcommand{\R}{\mathbb{R}}
\newcommand{\Z}{\mathbb{Z}}
\newcommand{\C}{\mathbb{C}}
\newcommand{\st}{\mid}  % "such that" in set-builder notation--can be chg'd

\newcommand{\cross}{\times}
\newcommand{\abs}[1]{\lvert#1\rvert}
\newcommand{\ip}[2]{\langle#1,#2\rangle}	% inner product

\newcommand{\define}[1]{\textsl{#1}}
\newcommand{\gap}{\vspace{2ex}}

\newcommand{\into}{\hookrightarrow}
\newcommand{\inv}{^{-1}} % "inverse"
\newcommand{\w}{\omega}	  %lazy symplectic form

\theoremstyle{definition}
\newtheorem{thm}{Theorem}

\newtheorem{claim}[thm]{Claim}

\newtheorem*{rmk}{Remark}
\newtheorem*{thmn}{Theorem}

\newtheorem{fact}{Fact}
\theoremstyle{remark}
\newtheorem*{eg}{Example}

\newcommand{\rn}{\R^n}
\newcommand{\rnn}{\R^N}
\newcommand{\DD}{\Delta}
\newcommand{\LL}{\mathbb{L}}
\newcommand{\OO}{\mathcal{O}}
\newcommand{\F}{\mathcal{F}}
\newcommand{\Q}{\mathcal{Q}}
\renewcommand{\k}{\Bbbk} 

\newcommand{\gq}{geometric quantization}
\DeclareMathOperator{\im}{im}

\begin{document}

\title[Quantization of toric manifolds]{The quantization of a toric 
manifold is given by the integer lattice points in the moment polytope}

\author{Mark D. Hamilton}
\thanks{Supported by a PIMS Postdoctoral Fellowship}
\address{Department of Mathematics and Statistics\\ University of Calgary\\
Calgary, AB\\ Canada}
\email{umbra@math.toronto.edu}
\date{July 27, 2007}

\subjclass[2000]{Primary 53D50}

\begin{abstract}
We describe a very nice argument, which we learned from Sue Tolman, 
that the dimension of the quantization space 
of a toric manifold, using a K\"ahler polarization,
is given by the number of integer lattice points in the moment polytope.
\end{abstract}

\maketitle

\section{Introduction}
``The quantization of a toric manifold is given by the integer 
lattice points in the moment polytope.'' 

In principle, 
this is a well-known result; nevertheless, it does not 
seem to be written down in exactly this language.
%(See, for example, Cor.\ 2.9 in~\cite{oda}.)
Usually reference is made to the paper~\cite{danilov} of Danilov, 
where this result is phrased in algebro-geometric terms, 
about the sheaf cohomology of a manifold 
(compare~\eqref{qndef}).
Guillemin, Ginzburg, and Karshon describe it 
as a ``folk theorem, usually attributed
to Atiyah or Danilov.'' (\cite{ggk}, p 142)   

A precise statement is as follows:
\begin{thmn}
Let $M$ be a toric $2n$-manifold, with moment polytope $\DD \subset \rn$.  
Then the dimension of the quantization is equal to the number of 
integer lattice points in $\DD$, that is,
\[ \dim H^0 (M,\OO_\LL) = \#(\DD\cap \Z^n). \]
\end{thmn}

There is a lovely proof of this fact, which
I learned from Sue Tolman, but 
as far as I know, it does not appear in the literature 
in this form.  
(In~\cite{ggk} they give a proof using these ideas (Proposition 8.4), 
but it is embedded in a much more general discussion and is not 
so easy to isolate.)
The argument is so straightforward and accessible I thought it worth
presenting on its own.  
I do not at all claim originality; 
rather, the aim of this paper is to present the argument
simply and clearly.
Thus, I have made no attempt to be as general as I can,
but have chosen transparency over generality wherever possible.

Delzant's construction of a toric manifold from 
its moment polytope is essential to this argument, and so we 
review it in Section~\ref{s:construct}.
The proof of the Theorem appears in Section~\ref{s:qn}, 
after stating a few facts about quantization and toric manifolds.
Finally, in Section~\ref{s:eg}, 
we show how the concepts in this paper apply in a simple example.

We assume the reader is familiar with the basic concepts of symplectic 
toric geometry, and thus we do not define terms like ``toric manifold'' 
and ``moment map.''  If these are unfamiliar, we recommend
the introduction by Cannas da Silva (\cite{CdaS}).

\noindent{\bf Acknowledgements.}  
I am grateful to Yael Karshon for explaining this argument to me in 
the first place several 
years ago, and for answering my questions more recently.
I am also grateful to Sue Tolman for taking the time to answer my 
questions as well.
Finally, I thank Paul Sloboda for his hospitality during the early stages of
work on this paper.

\section{Construction of toric manifolds}\label{s:construct}
We present here two different, 
though related, constructions of a toric manifold from its moment 
polytope, 
which we call the ``symplectic'' and ``complex'' constructions.
This is intended to be a review, and so we skip a number of details, 
including most of the proofs
(and so, in particular, this is not a good place to learn the 
constructions for the first time.
For that, the reader is directed to~\cite{CdaS} for the symplectic 
construction, and~\cite{kartol} and~\cite{audin} for the complex construction).

\subsection{Symplectic construction}\label{ss:delz}
This construction is due to Delzant (\cite{delzant}).
Given a convex polytope $\DD \subset \rn$, 
it produces a symplectic manifold $M^{2n}$, together with an effective 
action of the torus $T^n \cong (S^1)^n$, 
whose moment map image is precisely $\DD$.
The polytope is required to satisfy the condition that 
at each vertex there are $n$ edges, generated by a $\Z$-basis 
for the lattice $\Z^n$; such polytopes are often
called \emph{Delzant} polytopes.  
As shown in~\cite{delzant} (see the remark on p.\ 323), 
these are precisely the polytopes 
that appear as moment map images of toric manifolds.
Cannas da Silva gives a lovely explanation of Delzant's construction
in~\cite{CdaS}, which we follow to some extent, 
although we caution the reader that we use slightly different 
sign conventions than she does.

Let $\DD$ be a convex polytope in $\rn$ with $N$ facets (codimension-1 faces),
satisfying Delzant's condition.
For each facet of $\DD$, let $v_j \in \Z^n$ be the primitive\footnote{
a vector $v \in \Z^n$ is \define{primitive} if its coordinates have no 
common factor} %end footnote
inward-pointing vector normal to the facet.  
Define a projection $\pi$ from $\rnn$ to $\rn$ by taking the $j$th 
basis vector in $\rnn$ to $v_j \in \rn$:
\begin{equation}\label{map}
\begin{split}
\pi\colon \rnn &\to \rn \\
	 e_j &\mapsto v_j
\end{split}
\end{equation}
Delzant's condition on $\DD$ implies that the $v_j$ span $\rn$; 
in fact, the $n$ vectors normal to the facets meeting at any one 
vertex form a $\Z$-basis for $\Z^n$ 
(this is left as a linear algebra exercise for the reader,
although we note that this is how a Delzant polytope is defined 
in~\cite{ggk}).
Thus $\pi$ maps 
$\Z^N$ onto $\Z^n$ and so induces a map (which we also call $\pi$) 
between tori, 
\[ \pi \colon \rnn/\Z^N \to \rn/\Z^n. \]
Let $K$ be the kernel of this map, and $\k$ 
the kernel of the map~\eqref{map}, which will in fact be the Lie 
algebra of $K$.
We then get two exact sequences
\begin{subequations}\label{sequences}
\begin{align}
\begin{CD}\label{torseq}
1	@>>>	K	@>i>>	T^N	@>\pi>>	T^n	@>>> 1
\end{CD}\\
\begin{CD}\label{lieseq}
0	@>>>	\k	@>i>>	\rnn	@>\pi>>	\rn	@>>> 0
\end{CD}
\end{align}
\end{subequations}
and the dual exact sequence
\begin{equation}\label{dualseq}
\begin{CD}
0	@>>>	(\rn)^*	@>\pi^*>>  (\rnn)^*	@>i^*>L>  \k^*	@>>> 0	
\end{CD}
\end{equation}
with induced maps as shown.
Since we will be working a lot with $i^*$, we will denote it by $L$,
both for ease of notation and to emphasize that it is a linear map between 
vector spaces.

Using the vectors $v_j$, we can write the polytope as
\[ \DD = \{ x \in \rn \st \ip{x}{v_j} \geq \lambda_j, \; 1 \leq j \leq N \} \]
for some real numbers $\lambda_j$.
We assume that the $\lambda_j$ are all integers;
this will ensure that $M$ is pre-quantizable 
(see Fact~\ref{qbl} in \S\ref{s:qn}).
This gives us a vector $\lambda \in \Z^N$.

Let $\nu = L (-\lambda) \in \k^*$ (identifying $(\rnn)^*$ with $\rnn$).
Since the sequence~\eqref{dualseq} is exact, $L\inv(0) = \im \pi^*$,
so $L\inv(\nu) = \im (\pi^* - \lambda)$ 
and, since $L$ is a 
linear map between vector spaces, $L\inv(\nu)$  is an affine subspace 
of $\rnn$.
The intersection of this affine subspace with $\rnn_+$, the positive 
quadrant in $\rnn$, can be identified with the polytope $\DD$.
More precisely,

\begin{claim}\label{lattice}
Let $\DD' = L\inv(\nu) \cap \rnn_+$.
Then the map $\pi^* - \lambda$ restricts to an affine bijection from 
$\DD$ to $\DD'$, such that 
the integer lattice points in $\DD$ correspond to $\DD' \cap \Z^N_+$.
\end{claim}

\begin{proof}
For $x \in \rn$, $(\pi^*-\lambda)(x) \in \rnn_+$ if{f} $x\in\DD$, as follows:
\begin{equation}
\begin{split}
(\pi^* - \lambda) (x) \in \rnn_+ 
&\iff \ip{\pi^*(x)-\lambda}{e_j} \geq 0 \quad \forall j \\
&\iff \ip{\pi^*(x)}{e_j} - \lambda_j \geq 0\\
&\iff \ip{x}{\pi(e_j)} \geq \lambda_j\\
&\iff \ip{x}{v_j} \geq \lambda_j \quad \forall j\\
&\iff x \in \DD.
\end{split}
\end{equation}

Since $(\pi^* - \lambda)$ is an affine injection
$\rn \into \rnn$, it is a bijection onto its image,
and so it is a bijection from $\DD$ onto 
$\im(\pi^* -\lambda)\cap\rnn_+$,
which is $L\inv(\nu) \cap \rnn_+$ as argued above.

Finally, since all of the coordinates of each of the $v_j$ and $\lambda$
are integers,  $(\pi^* - \lambda)$ maps $\Z^n$ into $\Z^N$.
It only remains to see that if a point in $\DD'$ has integer coordinates, 
then it is the image under $(\pi^* - \lambda)$ of a point in $\Z^n$, 
for which it suffices to show that if $\pi^*(x)=y \in \Z^N$, then $x\in \Z^n$.

The map $\pi^*$ can be written as 
\( y = V x \)
where $y$ and $x$ are the variables in $\rnn$ and $\rn$, 
written as column vectors, and 
$V$ is the $N \cross n$ matrix whose rows are the vectors $v_j$.

As noted in the previous section, $n$ of the $v_j$s corresponding to one 
vertex form a $\Z$-basis of $\Z^n$; wolog suppose $v_1, \ldots, v_n$ form such 
a basis.  
Let $\bar{V}$ be the $n\cross n$ matrix whose rows are $v_1, \ldots, v_n$,
so that the equation $Y = \bar{V} x$ defines $y_1, \ldots, y_n$ from the $x$s 
(where $Y$ is the column vector of $y_1, \ldots, y_n$).

Since $v_1, \ldots, v_n$ form a $\Z$-basis for $\Z^n$, 
the determinant of $\bar{V}$ is $\pm 1$, 
so $\bar{V}$ is invertible and its inverse has integer entries.  
Thus, given a $Y$ with integer entries, 
$x = \bar{V}\inv Y$  will also have integer entries, 
and so integer lattice points in the image of $\pi^*$ come from 
integer lattice points in $\rn$.

\end{proof}

The torus $T^N$ acts on $\C^N$ in the standard way, 
by componentwise multiplication;
this action is Hamiltonian with moment map 
$\phi(z_1,\ldots,z_N) = (\pi\abs{z_1}^2,\ldots,\pi\abs{z_N}^2)$
(where here we mean the \emph{number} $\pi$, not the map from~
\eqref{sequences}).
The inclusion $i \colon K \into T^N$ induces 
a Hamiltonian action of $K$ on $\C^N$ with moment map 
$\mu = i^* \circ \phi$ from $\C^N \to \k^*$.

\[
\xymatrix{
&\rn \ar[d]^{\pi^*}\\
\C^N \ar[r]^\phi \ar@{.>}[rd]_\mu 
&\rnn \ar[d]^{L=i^*} \\
&\k^*}
\]

Let $M = \mu\inv(\nu)/K$.  
The action of $T^N$ on $\C^N$ commutes with the action of $K$
and thus 
descends to a Hamiltonian action on the quotient $M$.
This action is not effective; however, the quotient torus 
$T^n = T^N/K $ acts  effectively.
It is a theorem of Delzant that $M$ with this action 
is a smooth toric manifold, with moment polytope $\DD$.
(See \cite{delzant}, p 329
or~\cite{CdaS}, Claim 2 in section I.2.5.)

In summary, then, $M$ is the symplectic reduction of $\C^N$ 
by $K$ at $\nu \in \k^*$, 
where both $\nu$ and $K$ are determined by the polytope
(using the sequences~\eqref{sequences} and~\eqref{dualseq}).

\subsection{Complex construction}\label{ss:cplx}

The above construction 
produces $M$ as a symplectic manifold with a torus action,
but does not give the K\"ahler structure.
For that, we use a different construction
that realizes $M$ as the quotient of an open set $U_\F$ in $\C^N$ 
by the action of a complex torus $K_\C$.
This is usually (eg in~\cite{audin}, \cite{kartol}) described 
in terms of a \emph{fan.}  However, it can also be done directly from 
the polytope,
and we have chosen this approach to avoid having 
to explain the notion of fans.
Note that, even phrased in the language of fans, this approach is 
\emph{not} the same as the complex construction that appears in~\cite{CdaS}.
The latter is the algebraic geometry construction of $M$ 
as a toric variety, which is 
the same as given for example in~\cite{fulton} and~\cite{oda}.

Begin with a Delzant polytope $\DD$ as in the previous section, and construct
the map $\pi$ and the exact sequences~\eqref{sequences} and~\eqref{dualseq}
as described there.
If we complexify the sequence~\eqref{torseq}, we get an exact sequence
\begin{equation}\label{cplxseq}
\begin{CD}
1	@>>>	K_\C	@>i>>	T^N_\C	@>\pi>>	T^n_\C	@>>> 1
\end{CD}
\end{equation}
of complex tori, where $T^N_\C$ denotes the complex torus $(\C^*)^N$, 
and $K_\C = \ker i$ is the complexification of $K$.
  
Let $F_1$, $F_2$, \ldots, $F_N$ be the facets of $\DD$.
Define a family $\F$ of subsets of $\{ 1, 2, \ldots, N\}$ as follows:
\begin{itemize}
\item $\varnothing \in \F$
\item $ I \in \F \iff \bigcap_{j\in I} F_j \neq \varnothing$
\end{itemize}
i.e., $I=\{j_1,\ldots,j_k\}$ is in $\F$ if{f} 
the intersection $F_{j_1} \cap \cdots \cap F_{j_k}$ is non-empty.

The open set $U_\F \subset \C^N$ is constructed as follows.
Given a point $z = (z_1,\ldots,z_N) \in \C^N$, 
let its \define{zero-index set} be the set 
\[ I_z = \{ j \st z_j = 0 \}. \]
Then $U_\F$ is defined to be the set of $z$ in $\C^N$ whose 
zero-index sets are in $\F$,
\begin{equation}\label{usigma}
z \in U_\F \iff I_z \in \F.
\end{equation}
Then it is a theorem that $M = U_\F / K_\C$,
where $K_\C$ acts via the inclusion $i\colon K_\C \into T^N_\C$,
is a smooth toric manifold.
(See for example~\cite{audin}, Proposition VII.1.14 and surrounding discussion.
  Her $U_\Sigma$ is the same as our $U_\F$, 
although it is constructed differently.) 

\begin{eg}
Suppose $\DD$ is the unit square in $\R^2$, with facets numbered as shown
in~\ref{f:sqr}.
Then 
\[ \F = \bigl\lbrace \varnothing, \{1\}, \{2\}, \{3\}, \{4\}, \{1,2\}, \{2,3\},
\{3,4\}, \{1,4\} \bigr\rbrace. \]
The only points \emph{not} in $U_\F$ are those 
with first and third coordinate zero, or second and fourth coordinate zero, 
and so 
\[ U_\F = \C^4 \smallsetminus 
\Bigl( \{ z_1 = 0 = z_3 \} \cup \{ z_2 = 0 = z_4 \} \Bigr). \]
\begin{figure}
\centerline{\includegraphics[height=1.2in]{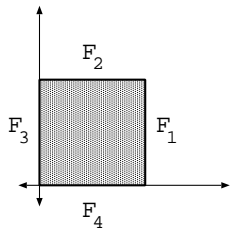}}
\caption{The polytope for the example}\label{f:sqr}
\end{figure}
\end{eg}

\begin{rmk}
Note that $\C^N \smallsetminus U_\F$ is the union of 
submanifolds of (complex) codimension at least 2, for the following reason: 
from \eqref{usigma}, 
$\C^N \smallsetminus U_\F$ will be the set of points 
whose zero-index set is \emph{not} in $\F$.
Since each of the singletons $\{ j \}$ are in $\F$, 
all points in $\C^N \smallsetminus U_\F$ 
have at least two coordinates which are zero.  
Thus $\C^N \smallsetminus U_\F$
is the union of sets of codimension at least~2.
\end{rmk}

The fact that these two constructions yield the same manifolds 
is Remark 2.6 in~\cite{kartol}.  
The reason is the following:
The sets $\mu\inv(\nu)$ and $U_\F$ are related as follows: 
\begin{equation}\label{uf}
 U_\F = K_\C \cdot \mu\inv(\nu)
\end{equation}
(see~\cite{ggk}, section 5.5.).
Thus it is no surprise that $U_\F / K_\C = \mu\inv(\nu) / K$.

\section{Quantization}\label{s:qn}

The purpose of geometric quantization is 
to associate, to a symplectic manifold $(M,\w)$, a Hilbert space 
(or a vector space) $\Q(M)$.  
(The terminology ``quantization'' comes from physics, 
where we think of $M$ as 
a classical mechanical system, and $\Q(M)$ as the 
space of wave functions of the corresponding quantum system.)
Much of the motivation for \gq\ in mathematics comes from 
representation theory.  

The basic building block in the geometric quantization of $(M,\omega)$ 
is a complex line bundle $\LL \to M$ with a connection 
whose curvature form is $\omega$,
called a \define{prequantum line bundle.}
If such an $\LL$ exists, $M$ is called \define{prequantizable,}
which will be the case if $[\omega] \in H^2(M,\R)$ is integral.
The quantization space $\Q(M)$ is constructed from sections of $\LL$.

The space of all sections is in general ``too big,'' and so a 
further structure, called a \emph{polarization,} is used to cut down 
on the number of sections.  
One of the most common is a K\"ahler polarization, which is a 
complex structure on $M$ compatible with the symplectic form.  
In this case, we take $\LL$ to be a holomorphic line bundle,
and the quantization space is the space of holomorphic sections of $\LL$ 
over $M$:
$\Q(M) = \Gamma_{\OO} (M,\LL) = H^0 (M,\OO_\LL)$.
This is what we use in the case of toric manifolds.

\begin{rmk} The quantization is often defined to be the virtual vector 
space 
\begin{equation}\label{qndef}
\Q(M) = \sum_{j\geq 0} (-1)^j H^j (M,\OO_\LL). 
\end{equation}
When $M$ is a toric manifold, 
all these cohomology groups with $j>0$ are zero 
(see e.g.~\cite{danilov}, Cor.\ 7.4, or~\cite{oda}, Cor.\ 2.8),
and so the quantization 
is just $H^0(M,\OO_\LL)$, namely the space of holomorphic sections.\\
\end{rmk}

(There are many sources for geometric quantization, for the reader who 
wishes more than these very sketchy details.  The books~\cite{wood} 
and~\cite{sniat} are classic references, if both rather technical; 
~\cite{puta} is perhaps easier as an introduction, though still very complete. 
The referee pointed me to~\cite{enriq}, available on the arXiv.
John Baez has a good brief introduction on the Web at~\cite{baez}.
\cite{ggk} also has a good, brief introduction at the beginning of 
Chapter 6, and refer to numerous other sources.  
There are of course many other references.)

We first state some facts about quantization applied to toric manifolds:

\begin{fact}\label{qbl}
The toric manifold $M$ constructed from $\DD$ is pre-quantizable 
if\footnote{
This ``if'' is actually an ``if{f},'' modulo some subtleties about 
equivariance.  
If $M$ is \emph{equivariantly} prequantizable, 
in the sense of~\cite{ggk} chapter 6, then it is necessary 
that $\lambda$ be in $\Z^N$.
If $M$ is ``non-equivariantly'' prequantizable, then the polytope 
(and thus $\lambda$) still satisfies an integrality condition.  
Since the moment map is defined only up to a constant, we can add a constant to
$\lambda$ without changing the construction of $M$;
this corresponds to translating $\DD$ without changing its shape.
The integrality condition implies that we can choose the constant 
so that $\lambda$ lies in $\Z^N$. } %end footnote
the $\lambda \in \R^N$ appearing in the symplectic construction 
is in $\Z^N$.
\end{fact}
See~\cite{ggk}, Example 6.10 on p.\ 93; see also~\cite{delzant}, p 327.

\begin{fact}
If the toric manifold $M$ is presented as $U_\F / K_\C$ 
as in \S\ref{ss:cplx}, then 
\begin{equation}\label{linebdl}
\LL = U_\F \cross_{K_\C} \C
\end{equation}
is a prequantum line bundle, 
where $K_\C$ acts on $\C$ with weight $\nu = L(-\lambda) \in \k^*$ 
($\lambda$ as in \S\ref{ss:delz}).
\end{fact}

Now we come to the main point of the paper, which is to give a proof 
of the following theorem:

\begin{thmn}\label{onlythm}
Let $M_\DD$ be a toric manifold, with moment polytope $\DD \subset \rn$.  
Then the dimension of the quantization space is equal to the number of 
integer lattice points in $\DD$, 
\[ \dim H^0 (M,\OO_\LL) = \#(\DD\cap \Z^n). \]
\end{thmn}

\begin{proof}
A holomorphic section of $\LL = (U_\F \cross \C) / K_\C$ 
over $M = U_\F / K_\C$ corresponds to a $K_\C$-equivariant, 
holomorphic function $s \colon U_\F \to \C$.
Because $\C^N \smallsetminus U_\F$ 
is the union of submanifolds of codimension 
greater than or equal to 2
(see the Remark in \S\ref{ss:cplx}),
$s$ extends to a holomorphic function on all of $\C^N$,
which we still call $s$.  
(This follows from Hartogs' theorem, for example~\cite{hartogs} p.\ 7 
--- a holomorphic function on $\C^N$ for $N>1$ cannot have an isolated 
singularity, and therefore cannot have a singularity on a submanifold 
of codimension $\geq 2$.)

Thus we are looking for the $K_\C$-equivariant, holomorphic functions 
$ s \colon \C^N \to \C$, 
where the action of $K_\C$ on $\C$ is with weight $\nu$, 
and the action on $\C^N$ is via the inclusion $i\colon K_\C \into T^N_\C$
and the standard action of $T^N_\C$ on $\C^N$.
Write such a function $s$ as its Taylor series, so that 
\[ s = \sum_{I\in \Z_+^N} a_I z^I \]
(where $I = \{ j_1, \ldots, j_N \}$ is a multi-index, 
$a_I$ is a complex number, 
and as usual in complex variables
$z^I$ means $z_1^{j_1} z_2^{j_2} \cdots z_N^{j_N}$).
Consider one term $z^I$ in this sum at a time.

First note that, for $t \in T^N_\C$ and $z \in \C^N$, 
\[ (t\cdot z)^I = (t_1 z_1, \ldots, t_N z_N )^I 
= \bigl( (t_1 z_1)^{j_1} \cdots (t_N z_N)^{j_N}\bigr) 
= t^I z^I. \]
Now suppose $s(z) = z^I$, and see when it is equivariant.
First, 
\[ s(k\cdot z) = s(i(k)\cdot z) = (i(k) \cdot z)^I = i(k)^I z^I 
= k^{i^*(I)} z^I. \]
On the other hand, 
\[ k \cdot s(z) = k^\nu \cdot z^I. \]
Thus $s(k\cdot z) = k\cdot s(z)$ when $i^*(I) = \nu$, i.e. $L(I) = \nu$.

Therefore, a basis for the space of equivariant sections, 
and thus for $H^0(M,\OO_\LL)$, 
is 
\[ \{ z^I \st L(I) = \nu, \; I \in \Z_+^N \}. \]
The set of such $I$ is $\Z^N_+ \cap L\inv(\nu)$, which, 
as noted in Claim~\ref{lattice},
corresponds precisely with 
the set of integer lattice points in the moment polytope $\DD$.
\end{proof}

\section{Example}\label{s:eg}

To see how all of these constructions play out, we will go through 
one example in detail (with some calculations left as exercises).
Take the polytope to be the triangle in $\R^2$ with vertices 
$(0,0)$, $(0,m)$, and $(m,0)$, for $m \in \Z_+$, as shown.
Here $N=3$ and $n=2$, so we will be constructing a 4-dimensional 
manifold as a quotient of $\C^3$, with an action of $T^2$.

\begin{figure}[ht]
\centerline{\includegraphics[height=1in]{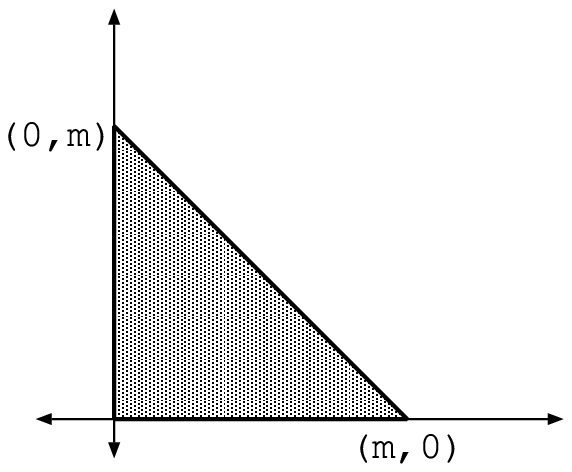}}
\caption{The polytope $\DD$ for this example}\label{f:polytope}
\end{figure}

The three normal vectors are 
\[ v_1 = ( 0,1 ) \quad v_2 = ( 1,0 ) \quad v_3 = ( -1 , -1 ) \]
and $\lambda$ is $(0,0,-m)$. 
Therefore the map $\pi \colon \R^3 \to \R^2$ in~\eqref{map}
can be written as the matrix 
\[
\begin{bmatrix}
0 & 1 & -1 \\
1 & 0 & -1 
\end{bmatrix}
\]
or, writing the coordinates in $\R^3$ as $(x,y,z)$, 
\[ \pi(x,y,z) = (y-z, x-z). \]
The kernel of this map is $ \{ x = y = z \} $ in $\R^3$, 
which is $\k$, which we identify with $\R$ by $i\colon t \mapsto (t,t,t)$.

The corresponding map on tori is 
\[ \bigl( e^{2\pi ix},e^{2\pi iy},e^{2\pi iz}\bigr) \mapsto 
\bigl( e^{2\pi i(y-z)},e^{2\pi i(x-z)}\bigr) \]
with kernel $K = (e^{2\pi it},e^{2\pi it},e^{2\pi it})$, 
which is $S^1$ embedded into $T^3$ as the diagonal subtorus.

For the dual sequence, the map $\pi^*$ will be given by the transpose matrix
\[ 
\begin{bmatrix}
1 & 0 \\  0 & 1  \\  -1 & -1
\end{bmatrix}
\]
or, writing coordinates in $\R^2$ as $(a,b)$, 
\[ \pi^*(a,b) = \bigl(a,b,(-a-b)\bigr). \] 
Similarly, the map $L=i^*$ is $L(x,y,z) = x+y+z$.

From this, we get that $\nu = L\bigl( -(0,0,-m)\bigr) = m$.
Therefore, the affine space $L\inv(\nu)$ is the space 
$\{ x + y + z = m \}$ lying in $\R^3$;
the intersection of this space
with the positive orthant $\R^3_+$ is a triangle, 
whose identification with $\DD$ is easy to see.
(See Figure~\ref{f:poly3d}.)

\begin{figure}[ht]
\centerline{\includegraphics{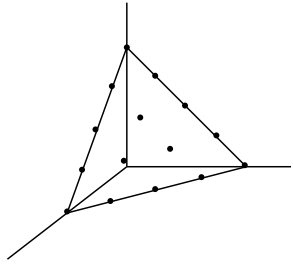}}
\caption{$L\inv(\nu)$ with integer lattice points}\label{f:poly3d}
\end{figure}

Pulling this intersection $L\inv(\nu) \cap \R^3_+$ back 
by the map $\phi \colon \C^3 \to \R^3$ gives us 
\[ \mu\inv(\nu) = \{ z \in \C^3 \st 
\abs{z_1}^2 + \abs{z_2}^2 + \abs{z_3}^2 = m/\pi \} \cong S^5; \]
the reduction of this by the diagonal action of $S^1$ is $\C P^2$.

Note that the integer points in $\DD'$
will be the set $\{ (x,y,z) \in \Z^3 \st x,y,z \geq 0, \; x+y+z = m \}$,
which is in one-to-one correspondence with the set 
\[ \{ (x,y) \in \Z^2 \st  x,y \geq 0, \; x+y \leq m \} \]
integer points in $\DD$.

For the complex construction,   
labelling the facets using the same numbering as we used for the normal 
vectors, the collection of subsets $\F$ corresponding to this polytope 
is 
\[ \F = \bigl\lbrace \varnothing, 
\{1\}, \{2\}, \{3\}, \{1,2\}, \{2,3\}, \{1,3\} \bigr\rbrace. \]
Thus $U_\F$ is the set of points in $\C^3$ which have either zero, 
one, or two coordinates zero, that is, $U_\F = \C^3 \smallsetminus \{ 0 \}$.

The complex torus here is the complexification of $K$, i.e.\ 
$K_\C = \C^*$, acting on $\C^3$ by the diagonal action.
The quotient of $\C^3 \smallsetminus \{ 0\}$ by the diagonal action 
of $\C^*$ is $\C P^2$.

(Notice that in passing to the complex construction we lose the information 
about the ``size'' of the reduced space.
This is a general phenomenon --- $U_\F$ ``remembers'' the directions of the 
faces of the polytope, but not their sizes.)

Finally, the prequantum line bundle will be $\LL = U_\F \cross_{K_\C} \C$,
where $K_\C$ acts on $\C$ with weight $m$
(namely $k \cdot z = k^m z$, $k \in \C^*$).

For the space of sections, we are looking for $K_\C$-equivariant 
holomorphic functions
$s \colon U_\F \to \C$, which, as noted in \S\ref{s:qn}, will extend 
to a $K_\C$-equivariant, holomorphic function $s\colon \C^3 \to \C$.

What are such functions? 
Take $s$ to be a monomial $z_1^{j_1} z_2^{j_2} z_3^{j_3}$.
For $k \in \C^* \cong K_\C$, 
\begin{equation}
\begin{split}
s(k\cdot z) &= (k z_1)^{j_1} (k z_2)^{j_2} (k z_3)^{j_3}\\
&= k^{(j_1 + j_2 + j_3)} z_1^{j_1} z_2^{j_2} z_3^{j_3},
\end{split}
\end{equation}
which will equal $k\cdot s(z) = k^m s(z)$ precisely when 
$j_1 + j_2 + j_3 = m$,
that is, the triple of integers  $(j_1,j_2,j_3)$ lies in $\DD'$.
Thus, the monomials $z^I$ such that $I \in \DD'$ will be a basis
for the space of $K_\C$-equivariant, holomorphic sections of $\LL$.

Looking at the polytope in Figure~\ref{f:polytope}, we can see 
that there will be 
$\mbox{$(m+1)$} + m + \cdots + 1$ points with integer coordinates,
and so the quantization will have dimension $\frac{m(m+1)}{2}$.

\gap
\noindent{\bf Exercise:} Repeat the above procedure when the polytope 
is the square considered in the example in \S\ref{ss:cplx}.

\end{document}